\documentclass[a4paper,12pt]{amsart}

\usepackage{cmap}
\usepackage[T1]{fontenc}
\usepackage[utf8]{inputenc}
\usepackage[english]{babel}
\usepackage{lmodern}
\usepackage{amsmath,amssymb,amsthm}
\usepackage{enumerate}
\usepackage{xurl}
\usepackage[hidelinks]{hyperref}
\newtheorem{theorem}{Theorem}[section]
\newtheorem{lemma}[theorem]{Lemma}
\newtheorem{proposition}[theorem]{Proposition}
\newtheorem{cor}[theorem]{Corollary}
\theoremstyle{definition}
\newtheorem{definition}[theorem]{Definition}
\newtheorem{example}[theorem]{Example}
\theoremstyle{remark}
\newtheorem{remark}[theorem]{Remark}

\numberwithin{equation}{section}

\begin{document}
\title[On some $D(u^2)$-triples and quadruples of triangular numbers]
{On some $D(u^2)$-triples and quadruples of triangular numbers}
\author[Pavao Radi\'{c}]{Pavao Radi\'{c}}
\date{September 18, 2026}

\begin{abstract}
We classify all $D(u^2)$-triples of triangular numbers whose smallest
element is $T_{4u-1}$ or $T_{4u}$. We study the regularity of
$D(N)$-triples of triangular numbers. We also show that infinitely many
positive integers $u$ admit an all-triangular $D(u^2)$-quadruple.
For every positive integer $k$, triangular number
$T_{16k^2-1}$ belongs to infinitely many $D(u^2)$-quadruples of
triangular numbers with distinct positive integer parameters $u$.
\end{abstract}
\maketitle



\section{Introduction}\label{intr}
\begin{definition}
Let $N\neq0$ be an integer. We call a set of $m$ distinct positive
integers a $D(N)$-$m$-tuple, or an Diophantine $m$-tuple with the property $D(N)$,
if the product of any two distinct elements, increased by $N$, is a
perfect square.
\end{definition}

The case $N=1$ is the classical one. In this case, such sets are usually
called Diophantine $m$-tuples. The first Diophantine quadruple in
integers, $ \{1,3,8,120\}$, was found by Fermat. Since then, one of the central questions in the theory of Diophantine
$m$-tuples has been how large such sets can be. It is now known that
there are no $D(1)$-quintuples and no $D(4)$-quintuples
\cite{HeTogbeZiegler2019,BliznacFilipin2019}.

We recall the notion of regularity. Let $N$ be a positive integer,
and let $\{a,b,c\}$ be a $D(N)$-triple. The triple is called regular if
\begin{equation}
c=a+b\pm2\sqrt{ab+N}.
\end{equation}
For square $N$ and any $D(N)$-triple $\{a,b,c\}$ we can define, as seen in \cite[p.~171, equation~(2.7)]{AHS1980}, two integers
\begin{equation}\label{eq:regularity2}
d_\pm=a+b+c+\frac{2}{N}\left(abc\pm\sqrt{(ab+N)(ac+N)(bc+N)}\right).
\end{equation}
Then $\{a,b,c,d\}$ is a $D(N)$-quadruple, called regular.
Equivalently, as seen in \cite{Gibbs2010}, it satisfies
\begin{equation}\label{eq:regular-quadruple}
N(a+b-c-d)^2=4(ab+N)(cd+N).
\end{equation}
Since $d_+>0$ and
$d_+d_-=(c-a-b)^2-4(ab+N)$, we have $d_-=0$ if and only if
the triple $\{a,b,c\}$ is regular. If $a<b<c$, then $d_-<c<d_+$.
Consequently, a $D(N)$-quadruple with $a<b<c<d$ is regular if and
only if $d=d_+$. It remains an open question whether every extension
of a $D(1)$-triple or a $D(4)$-triple by a larger element is regular.

For positive integers $u$, we investigate $D(u^2)$-$m$-tuples whose
elements are triangular numbers
\[
        T_n=\frac{n(n+1)}{2},\qquad n\in\mathbb{N}.
\]
The connection between Diophantine $m$-tuples and triangular numbers
was first explored in \cite{Deshpande}. Recently, Bliznac Trebje\v{s}anin
\cite{BliznacTriangular} constructed infinitely many $D(u^2)$-triples
of triangular numbers containing a prescribed triangular number.
Related questions concerning the more general property $D(a)$ were
first studied by the same author in \cite{BliznacTriangular2}, and
subsequently in \cite{BagciZhouZheng}, using Pellian equations.
We adopt the same approach here.

Let $n,a,N$ be positive integers. If $\{T_n,T_a\}$ is a $D(N)$-pair,
then there exists $r\in\mathbb N$ such that
\[
        T_nT_a+N=r^2.
\]
Substituting $T_n=n(n+1)/2$ and $T_a=a(a+1)/2$, and setting
\[
        x=4r,\qquad y=2a+1,
\]
we obtain the generalized Pellian equation
\[
        x^2-n(n+1)y^2=16N-n(n+1).
\]
Thus, the problem of determining triangular numbers $T_a$ such that
$\{T_n,T_a\}$ is a $D(N)$-pair reduces to the study of solutions of
this equation. As discussed in
\cite{BliznacTriangular,BliznacTriangular2}, depending on $n$ and $N$,
the equation may have no integer solutions, or it may have several
nonassociated fundamental solutions, yielding several distinct
sequences of solutions. If $N=u^2$, the equation has a solution for
every positive integer $n$ \cite[Theorem~1]{BliznacTriangular}.

In Section~\ref{sec:triples}, we focus on the two families of
$D(u^2)$-pairs with fixed first element, $\{T_{4u-1},T_a\}$ and $\{T_{4u},T_a\}$.
We describe all such pairs and show that, in each case, positive integers $a$ come from the single sequence of positive integers. Then we give the complete classification of all $D(u^2)$-triples
of triangular numbers whose smallest element is
$T_{4u-1}$ or $T_{4u}$. More precisely, we prove that all $D(u^2)$-triples $\{T_n,T_a,T_b\}$ where $n \in \{4u-1,4u\}$ are the ones where $a$ and $b$ are consecutive terms of the same sequence mentioned above to describe all pairs.

In Section~\ref{sec:regularity}, we prove that, for every fixed
positive integer $N$, there are only finitely many regular
$D(N)$-triples of triangular numbers. The resulting bounds also
imply that all triples classified in Section~\ref{sec:triples} are
irregular. We further identify values of $N$ for which no regular
$D(N)$-triples of triangular numbers exist.

Another reason for considering square parameters $N$ in
Sections~\ref{sec:triples} and~\ref{sec:quadruples} is that $8u^2$, which is not a triangular number,
forms a $D(u^2)$-pair with every positive triangular number $T_n$,
since
\begin{equation}\label{glavnajednakost}
    8u^2T_n+u^2=u^2(2n+1)^2.
\end{equation}
Consequently, any $D(u^2)$-quadruple consisting entirely of
triangular numbers can be extended to a $D(u^2)$-quintuple by
adjoining $8u^2$. This generalizes the case $u=1$ considered in
\cite{BagciZhouZheng}. The known nonexistence of $D(1)$-quintuples
therefore implies that there are no $D(1)$-quadruples all of whose
elements are triangular numbers
\cite[Proposition~3]{BagciZhouZheng}. The same argument applies to
$u=2$, since there are no $D(4)$-quintuples
\cite[Theorem~1]{BliznacFilipin2019}, which proves the following Proposition:
\begin{proposition}
There are no $D(4)$-quadruples all of whose elements are triangular
numbers.
\end{proposition}
However, such nonexistence does not hold for all positive integers
$u$. In Section~\ref{sec:quadruples}, we show that infinitely many
positive integers $u$ admit an all-triangular $D(u^2)$-quadruple.
More precisely, for every positive integer $k$, the fixed triangular
number $T_{16k^2-1}$ belongs to infinitely many such quadruples with
distinct positive integer parameters $u$. This is expressed in the
following main theorem, proved in Section~\ref{sec:quadruples}.
\begin{theorem}\label{thm:triangular-quadruples}
Let $k$ be a positive integer, and define the sequence
$(u_n)_{n\ge0}$ by
\[
u_0=k,\qquad u_1=k(32k^2-4k-1),
\]
and
\[
u_{n+2}=(64k^2-2)u_{n+1}-u_n
\qquad(n\ge0).
\]
Then, for every $n\ge2$, the set
\[
\left\{
T_{16k^2-1},\;
T_{\frac{u_n-4k^2-k}{2k}},\;
T_{\frac{u_n+4k^2-k}{2k}},\;
T_{\frac{u_n^2}{k^2}-1}
\right\}
\]
is a $D(u_n^2)$-quadruple consisting of four distinct positive
triangular numbers.
\end{theorem}

\section{Triples with smallest element \texorpdfstring{$T_{4u-1}$ or $T_{4u}$}{T(4u-1) or T(4u)}}\label{sec:triples}
Let $u$ be a positive integer and $n\in\{4u-1,4u\}$. If
$\{T_n,T_a\}$ is a $D(u^2)$-pair for a positive integer $a$, then
there exists $r\in\mathbb N$ such that
\[
T_nT_a+u^2=r^2.
\]
Substituting $T_n=n(n+1)/2$ and $T_a=a(a+1)/2$, and setting
\[
x=4r,\qquad y=2a+1,
\]
we obtain the generalized Pellian equation
\begin{equation}\label{eq:pell-minus}
x^2-n(n+1)y^2=16u^2-n(n+1).
\end{equation}
Put $D=n(n+1)$ and $U=2n+1$. The associated Pell equation
$X^2-DY^2=1$ has fundamental solution $(U,2)$.
The standard bounds for fundamental solutions $(x^*,y^*)$ of
\eqref{eq:pell-minus} \cite[Theorem~10.21]{knjiga} give
\[
\begin{cases}
0\le y^*\le1, & n=4u-1,\\
0<y^*\le1, & n=4u.
\end{cases}
\]
Thus, up to signs, the fundamental solutions are $(4u,1)$ and,
only when $n=4u-1$ and $u$ is a perfect square, $(2\sqrt u,0)$.
It is easy to see that the latter generates only solutions with even $y$. 
Hence this entire class is excluded by the condition $y=2a+1$.
The two choices $\pm4u+\sqrt D$ are related, up to an overall sign,
by multiplication by the $U+2\sqrt{D}$, since
\[
(-4u+\sqrt D)(U+2\sqrt D)=(U-8u)(4u+\sqrt D),
\quad U-8u\in\{-1,1\}.
\]
Thus, they do not generate distinct sequences of positive solutions. Consequently, all positive solutions of \eqref{eq:pell-minus} with
odd $y$ are given by
\begin{equation}
x_k+y_k\sqrt D=(4u+\sqrt D)(U+2\sqrt D)^k,
\quad k\ge0.
\end{equation}
From this, we get that the sequence $(y_k)_{k\ge0}$ satisfies
\begin{equation}\label{eq:yminus-rec}
y_0=1,\qquad y_1=8u+U,\qquad
y_{k+2}=2Uy_{k+1}-y_k\quad(k\ge0).
\end{equation}
Setting
\begin{equation}\label{eq:aminus-def}
a_k=\frac{y_k-1}{2},
\end{equation}
we obtain
\begingroup
\small
\begin{equation}\label{eq:aminus-rec}
a_1=n+4u,\quad
a_{k+2}=2Ua_{k+1}-a_k+2n\quad(k\ge0).
\end{equation}
\endgroup
Thus, the positive indices $a$ for which $\{T_n,T_a\}$ is a
$D(u^2)$-pair are precisely $a_k$ for $k\ge1$.
Moreover, $a_k>n$ for every $k\ge1$.

The following two lemmas follow from
\eqref{eq:yminus-rec} and can be proven by mathematical induction using the congruence $U^2\equiv1\pmod{8u}$.
\begin{lemma}[{\cite[Lemma~4]{BliznacTriangular2}}]\label{lem:minus-identities}
For all $k\ge0$,
\begin{equation}\label{eq:minus-quadratic}
y_{k+1}^2-2Uy_ky_{k+1}+y_k^2=4(16u^2-n(n+1)).
\end{equation}
\end{lemma}
\begin{lemma}\label{lem:minus-quadratic}
For all $k\ge0$,
\[
y_k\equiv U^k\pmod{8u}.
\]
\end{lemma}
It follows that $y_k^2\equiv1\pmod{8u}$. Since
\[
T_{a_k}=\frac{a_k(a_k+1)}{2}
=\frac{(2a_k+1)^2-1}{8}
=\frac{y_k^2-1}{8},
\]
we have $T_{a_k}=uA_k$, where $A_k=(y_k^2-1)/(8u)$ is a
nonnegative integer, positive for $k\ge1$.

We can now classify all $D(u^2)$-triples of triangular numbers whose
smallest element is $T_{4u-1}$ or $T_{4u}$.
\begin{theorem}\label{thm:minus-main}
Let $u\in\mathbb N$, $n\in\{4u-1,4u\}$ and $U=2n+1$, and let
$(a_k)_{k\ge0}$ be defined by \eqref{eq:aminus-rec}. Then:
\begin{enumerate}
\item For every $k\ge1$, the set
\[
\{T_n,T_{a_k},T_{a_{k+1}}\}
\]
is a $D(u^2)$-triple.
\item Conversely, if $a<b$ and $\{T_n,T_a,T_b\}$ is a
$D(u^2)$-triple of triangular numbers, then
\[
a=a_k,\qquad b=a_{k+1}
\]
for a unique $k\ge1$.
\end{enumerate}
\end{theorem}
\begin{proof}
Using Lemma~\ref{lem:minus-identities}, we obtain
\[
T_nT_{a_k}+u^2=\left(\frac{y_{k+1}-Uy_k}{8}\right)^2
\]
and
\[
T_{a_k}T_{a_{k+1}}+u^2
=\left(\frac{y_ky_{k+1}-U}{8}\right)^2.
\]
The quotients on the right-hand sides are integers by
Lemma~\ref{lem:minus-quadratic} and the congruence
$U^2\equiv1\pmod{8u}$. This proves the first assertion.

Conversely, let $a<b$ and suppose that $\{T_n,T_a,T_b\}$ is a
$D(u^2)$-triple. Since $\{T_n,T_a\}$ and $\{T_n,T_b\}$ are
$D(u^2)$-pairs, there exist $1\le i<j$ such that $a=a_i$ and $b=a_j$.
It remains to prove that $j=i+1$. By \eqref{glavnajednakost}, the set
\begin{equation}\label{cetvorka}
\{T_n,8u^2,T_{a_i},T_{a_j}\}
\end{equation}
has the property $D(u^2)$.
If $n=4u-1$, then $U=8u-1$, $T_n=u(U-1)$ and
$8u^2=u(U+1)$. If $n=4u$, then $U=8u+1$,
$8u^2=u(U-1)$ and $T_n=u(U+1)$.
Thus, every element of \eqref{cetvorka} is divisible by $u$.
Dividing these elements by $u$ gives a $D(1)$-quadruple
\[
\{U-1,U+1,A_i,A_j\}.
\]
The sequence $(A_k)_{k\ge0}$ is strictly increasing, and
$A_i\ge A_1=4U>U+1$ for $i\ge1$.
By \cite{Fujita} and \cite[p.~334, following Theorem~1]{Dujella}, every Diophantine
quadruple containing $\{U-1,U+1\}$ is regular.
Hence, by \eqref{eq:regular-quadruple},
\begin{equation}\label{eq:regular-minus}
(2U-A_i-A_j)^2=4U^2(A_iA_j+1).
\end{equation}
As a quadratic equation in $X=A_j$, this becomes
\begin{equation}\label{eq:quadratic-minus}
X^2-\bigl((4U^2-2)A_i+4U\bigr)X+A_i^2-4UA_i=0.
\end{equation}
The recurrence \eqref{eq:yminus-rec} and
Lemma~\ref{lem:minus-identities} yield
\begin{align}
A_{i-1}+A_{i+1}&=(4U^2-2)A_i+4U,\label{eq:Bminus-sum}\\
A_{i-1}A_{i+1}&=A_i^2-4UA_i.\label{eq:Bminus-product}
\end{align}
Therefore, the roots of \eqref{eq:quadratic-minus} are precisely
$A_{i-1}$ and $A_{i+1}$. Since $j>i$ and $(A_k)_{k\ge0}$ is strictly
increasing, we have $A_j=A_{i+1}$, and hence $j=i+1$.
\end{proof}

\begin{example}
For $u=1$ and $n=3$, we have
\[
(a_k)_{k\ge0}=(0,7,104,1455,20272,\ldots).
\]
Thus, the first few $D(1)$-pairs of triangular numbers with smallest
element $T_3$ are
\[
\{\mathbf{T_3},T_7\},\quad
\{\mathbf{T_3},T_{104}\},\quad
\{\mathbf{T_3},T_{1455}\},\quad
\{\mathbf{T_3},T_{20272}\},\quad\ldots
\]
By Theorem~\ref{thm:minus-main}, the first few $D(1)$-triples of
triangular numbers with smallest element $T_3$ are
\[
\{\mathbf{T_3},T_7,T_{104}\},\quad
\{\mathbf{T_3},T_{104},T_{1455}\},\quad
\{\mathbf{T_3},T_{1455},T_{20272}\},\quad\ldots
\]
\end{example}

The classification immediately rules out any additional triangular
element extending one of these triples to a quadruple.
\begin{cor}\label{korolar}
Let $u\in\mathbb N$. There is no $D(u^2)$-quadruple consisting only
of triangular numbers that contains $T_{4u-1}$ or $T_{4u}$.
\end{cor}
\begin{proof}
Let $n\in\{4u-1,4u\}$. Assume that
\[
\{T_n,T_a,T_b,T_c\},\qquad a<b<c,
\]
is a $D(u^2)$-quadruple of triangular numbers.
Applying Theorem~\ref{thm:minus-main} to $\{T_n,T_a,T_b\}$, we find
a unique $i\ge1$ such that $a=a_i$ and $b=a_{i+1}$.
Applying the same theorem to $\{T_n,T_a,T_c\}$ gives
$a=a_i$ and $c=a_{i+1}$. Thus $b=c$, a contradiction.
\end{proof}

\section{Regularity of triples of triangular numbers}\label{sec:regularity}
Throughout this section, $N$ is a positive integer and $T_n=n(n+1)/2$.
A $D(N)$-triple $\{T_a,T_b,T_c\}$, where $1\le a<b<c$, is regular if
$T_c=T_a+T_b+2\sqrt{T_aT_b+N}$.
The minus sign cannot occur for the largest element, since
$T_a+T_b-2\sqrt{T_aT_b+N}<T_b$.
Equivalently, regularity is expressed by
\begin{equation}\label{eq:triangular-regularity}
(T_c-T_a-T_b)^2=4(T_aT_b+N).
\end{equation}

\begin{theorem}\label{thm:finite-regular-triangular}
For every fixed positive integer $N$, there are only finitely many
regular $D(N)$-triples of triangular numbers. More precisely, if
$\{T_a,T_b,T_c\}$ is such a triple with $1\le a<b<c$, then
$c\ge a+b+1$ and
\begin{equation}\label{eq:regular-index-bounds}
(a+1)(b+1)(a+b+1)\le4N.
\end{equation}
\end{theorem}
\begin{proof}
Set $E=T_c-T_a-T_b=2\sqrt{T_aT_b+N}$. Since $N>0$, we have
\[
E>\sqrt{a(a+1)b(b+1)}>ab.
\]
Thus $T_c>T_a+T_b+ab=T_{a+b}$, which implies $c\ge a+b+1$.
Consequently,
\[
E\ge T_{a+b+1}-T_a-T_b=(a+1)(b+1).
\]
Using \eqref{eq:triangular-regularity}, we obtain
\[
\begin{aligned}
4N&=E^2-a(a+1)b(b+1)\\
&\ge(a+1)^2(b+1)^2-a(a+1)b(b+1)\\
&=(a+1)(b+1)(a+b+1),
\end{aligned}
\]
proving \eqref{eq:regular-index-bounds}.

Finally, $a\ge1$ and \eqref{eq:regular-index-bounds} yield
$(b+1)(b+2)\le2N$. Hence $b$ is bounded in terms of $N$, and
$1\le a<b$ leaves only finitely many pairs $(a,b)$. For each pair,
regularity determines $T_c=T_a+T_b+2\sqrt{T_aT_b+N}$ uniquely.
This proves the finiteness assertion.
\end{proof}

\begin{cor}\label{cor:two-families-irregular}
Let $u$ be a positive integer. Every $D(u^2)$-triple of triangular
numbers whose smallest element is $T_{4u-1}$ or $T_{4u}$ is irregular.
In particular, all triples classified in Theorem~\ref{thm:minus-main}
are irregular.
\end{cor}
\begin{proof}
Suppose that such a triple $\{T_a,T_b,T_c\}$, with $a<b<c$, is
regular. Its smallest index is $a=4u-1$ or $a=4u$, so in either case
$a\ge4u-1$. Since $b\ge a+1$, applying
\eqref{eq:regular-index-bounds} with $N=u^2$ yields
\[
\begin{aligned}
4u^2&\ge(a+1)(b+1)(a+b+1)\\
&\ge2(a+1)^2(a+2)\\
&\ge32u^2(4u+1).
\end{aligned}
\]
Thus $8(4u+1)\le1$, contradicting $u\ge1$.
\end{proof}

\begin{theorem}\label{thm:regular-congruence-obstructions}
\label{prop:regular-mod-three}
Let $N$ be a positive integer. There is no regular $D(N)$-triple of
triangular numbers if $N\equiv2\pmod3$.
\end{theorem}
\begin{proof}
Triangular numbers are congruent to $0$ or $1$ modulo $3$.
Substituting these residues into \eqref{eq:triangular-regularity}
gives $N\equiv0,1\pmod3$.
\end{proof}

In the following example, we use Theorem~\ref{thm:finite-regular-triangular}
to determine all regular $D(N)$-triples of triangular numbers for small $N$.
\begin{example}\label{ex:small-parameters}
\label{prop:small-square-parameters}
All regular $D(N)$-triples of triangular numbers for $1\le N\le36$
are listed below:
\[
\begin{array}{c|l}
N & \text{regular triples}\\ \hline
6  & \{1,3,10\}\\
10 & \{1,6,15\}\\
15 & \{1,10,21\}\\
18 & \{3,6,21\}\\
21 & \{1,15,28\}\\
28 & \{1,21,36\}\\
36 & \{1,28,45\},\quad\{3,15,36\}.
\end{array}
\]
\end{example}

\section{Quadruples of triangular numbers}\label{sec:quadruples}

The idea is to start with a $D(u^2)$-pair
$\{T_p,T_q\}$, where $u=(T_q-T_p)/2$ and $q=p+4k$ for some positive integer
$k$. We then extend this pair to a triple by adjoining $8u^2$ and consider its
two regular extension to a quadruple with elements $d_\pm$. This construction ensures that $d_\pm$ are are also triangular numbers. For easier notation, we express these numbers in terms of $u$ and $k$ and give the following proposition. We also need some additional conditions, the need of which is explained in the beginning of the proof.

\begin{proposition}\label{prop:triangular-construction}
Let $k$ and $u$ be positive integers satisfying
\begin{equation}\label{eq:construction-conditions}
u\equiv k\pmod{2k},
\qquad
u\ge k(32k^2+4k+1).
\end{equation}
Then the set
\begin{equation}\label{eq:construction-quadruple}
\left\{
T_{16k^2-1},\;
T_{\frac{u-4k^2-k}{2k}},\;
T_{\frac{u+4k^2-k}{2k}},\;
T_{\frac{u^2}{k^2}-1}
\right\}
\end{equation}
is an all-triangular $D(u^2)$-quadruple if and only if
there exists an integer $x$ satisfying
\begin{equation}\label{eq:construction-pell}
x^2-(64k^2-4)\frac{u^2}{k^2}=5-64k^2.
\end{equation}
\end{proposition}

\begin{proof}
The congruence in \eqref{eq:construction-conditions}
implies that $u/k$ is an odd integer. Consequently,
all four indices in \eqref{eq:construction-quadruple}
are integers.

Moreover, it is easy to see that the lower bound in
\eqref{eq:construction-conditions} gives
\[
0<16k^2-1
<
\frac{u-4k^2-k}{2k}
<
\frac{u+4k^2-k}{2k}
<
\frac{u^2}{k^2}-1.
\]
Thus the four triangular numbers are positive and
pairwise distinct.

Set
\[
a=T_{\frac{u-4k^2-k}{2k}},
\qquad
b=T_{\frac{u+4k^2-k}{2k}},
\qquad
c=8u^2.
\]
Since
\[
a=\frac{(u/k-4k)^2-1}{8},
\qquad
b=\frac{(u/k+4k)^2-1}{8},
\]
we have $b-a=2u$. Hence
\[
ab+u^2
=
ab+\frac{(b-a)^2}{4}
=
\left(\frac{a+b}{2}\right)^2.
\]
Identity~\eqref{glavnajednakost} also gives
\[
ac+u^2
=
\left[u\left(\frac{u}{k}-4k\right)\right]^2,
\qquad
bc+u^2
=
\left[u\left(\frac{u}{k}+4k\right)\right]^2.
\]
Therefore $\{a,b,c\}$ is a $D(u^2)$-triple. 

Consider the two regular extension values
\eqref{eq:regularity2}. Substituting our parameters, we obtain
\[
d_{\pm}
=
a+b+8u^2+16ab
\pm
(a+b)\left(\frac{u^2}{k^2}-16k^2\right).
\]
Using
\[
a+b=\frac{u^2/k^2+16k^2-1}{4}
\]
and
\[
16ab
=
\frac{(u^2/k^2+16k^2-1)^2-64u^2}{4},
\]
we find
\[
d_{\pm}
=
\frac{u^2/k^2+16k^2-1}{4}
\left[
\frac{u^2}{k^2}+16k^2
\pm\left(\frac{u^2}{k^2}-16k^2\right)
\right]
-8u^2.
\]
Consequently,
\begin{align*}
d_-
&=
8k^2(16k^2-1)
=
T_{16k^2-1},\\
d_+
&=
\frac{u^2(u^2-k^2)}{2k^4}
=
T_{\frac{u^2}{k^2}-1}.
\end{align*}

Thus five of the six required square conditions for
$\{d_-,a,b,d_+\}$ hold. The remaining
condition concerns the pair $d_-,d_+$. We have
\begin{align*}
d_-d_++u^2
&=
u^2\left[
4(16k^2-1)\left(\frac{u^2}{k^2}-1\right)+1
\right]\\
&=
u^2\left[
(64k^2-4)\frac{u^2}{k^2}+5-64k^2
\right].
\end{align*}
Since $u$ is a positive integer and the expression in
brackets is an integer, this is an integer square
if and only if the expression in brackets is an
integer square.
The remaining condition is therefore precisely
\eqref{eq:construction-pell}.
\end{proof}

We now prove Theorem~\ref{thm:triangular-quadruples} by constructing an
infinite sequence of solutions of \eqref{eq:construction-pell} and verifying
\eqref{eq:construction-conditions} for every $n\ge2$.
\begin{proof}[Proof of Theorem~\ref{thm:triangular-quadruples}]
Fix a positive integer $k$. First notice that the integer $64k^2-4$ is not a square, since
\[
(8k-1)^2<64k^2-4<(8k)^2.
\]
The associated Pell equation $x^2-(64k^2-4)\frac{u^2}{k^2}=1$ has fundamental solution $(32k^2-1,4k)$ (since $\sqrt{64k^2-4}=[8k-1,\overline{1,4k-2,1,16k-2}]$). 
The standard bounds for fundamental solutions $(x^*,(\frac{u}{k})^*)$ of (\ref{eq:construction-pell}) \cite[Theorem~10.21]{knjiga}  are then $|x^*|\leq\sqrt{(16k^2-1)(64k^2-5)}<32k^2-2$ and $0\leq y^* \leq 4k\sqrt{\frac{64k^2-5}{64k^2-4}}<4k$.

Thus, one fundamental solution is $(-1,1)$, from which we get the sequence of solutions $(x_n,\frac{u_n}{k})_{n\geq0}$ by
\begingroup
\small
\[x_n+\frac{u_n}{k}\sqrt{64k^2-4}=(-1+\sqrt{64k^2-4})(32k^2-1+4k\sqrt{64k^2-4})^n,\quad n\geq0.\]
\endgroup
    From this, if we put $\epsilon=32k^2-1+4k\sqrt{64k^2-4}$ and since $\epsilon + \epsilon^{-1}=64k^2-2$, values $(u_n)_{n\geq0}$ can be expressed by recurrence relation:
\begin{equation}
\begin{aligned}
u_0&=k,\qquad u_1=k(32k^2-4k-1),\\
u_{n+2}&=(64k^2-2)u_{n+1}-u_n\qquad(n\ge0).
\end{aligned}
\end{equation}
It remains to verify the conditions \eqref{eq:construction-conditions}.
First, the integers
\[
\frac{u_0}{k}=1,
\qquad
\frac{u_1}{k}=32k^2-4k-1
\]
are odd. Since $64k^2-2$ is even, by mathematical induction it is easy to verify that $u_n/k$ is an odd integer for every $n\ge0$. Consequently,
\begin{equation}\label{eq:family-congruence}
u_n\equiv k\pmod{2k}
\qquad(n\ge0).
\end{equation}

Second, $u_1>u_0>0$. If we assume that $u_{n+1}>u_n>0$ holds, then
\[
u_{n+2}
=
(64k^2-2)u_{n+1}-u_n
>
(64k^2-3)u_{n+1}
>
u_{n+1}.
\]
Thus $(u_n)_{n\ge0}$ is strictly increasing.

The recurrence gives
\[
\frac{u_2}{k}
=
2048k^4-256k^3-128k^2+8k+1.
\]
It follows that
\begin{align*}
\frac{u_2}{k}-(32k^2+4k+1)
&=
4k(512k^3-64k^2-40k+1)\\
&\ge
4k(408k^3+1)>0,
\end{align*}
where we used $k\ge1$. By monotonicity,
\begin{equation}\label{eq:family-lower-bound}
u_n\ge k(32k^2+4k+1)
\qquad(n\ge2).
\end{equation}

Equations \eqref{eq:family-congruence} and \eqref{eq:family-lower-bound} show that every $u_n$ with $n\ge2$ satisfies \eqref{eq:construction-conditions}. 
Thus Proposition~\ref{prop:triangular-construction} applies with $u=u_n$
and gives the required all-triangular $D(u_n^2)$-quadruple for every $n\ge2$.
\end{proof}

\begin{remark}
The restriction $n\geq 2$ excludes the two degenerate initial cases.
Indeed, $u_0=k$ gives $T_{u_0^2/k^2-1}=T_0=0$, whereas
\[
  \frac{u_1+4k^2-k}{2k}=16k^2-1,
\]
so for $n=1$ two of the displayed triangular numbers coincide.
\end{remark}

\begin{example}
Fix $k=1$. Then $u_0=1$, $u_1=27$, and
$u_{n+2}=62u_{n+1}-u_n$. For $n=2,3,4$, Theorem~\ref{thm:triangular-quadruples} gives
the following $D(u_n^2)$-quadruples:
\[
\renewcommand{\arraystretch}{1.3}
\begin{array}{c|r|l}
 n & u_n & \text{Quadruple} \\ \hline
 2 & 1673 & \{\mathbf{T_{15}},T_{834},T_{838},T_{2798928}\} \\
 3 & 103699 & \{\mathbf{T_{15}},T_{51847},T_{51851},T_{10753482600}\} \\
 4 & 6427665 & \{\mathbf{T_{15}},T_{3213830},T_{3213834},T_{41314877352224}\}
\end{array}
\]
\end{example}

\begin{proposition}
Every $D(u^2)$-quadruple obtained in Proposition~\ref{prop:triangular-construction} is irregular.
Consequently, all quadruples in Theorem~\ref{thm:triangular-quadruples} are irregular.
\end{proposition}

\begin{proof}
Using the notation of Proposition~\ref{prop:triangular-construction}, we have $d_-<a<b<c<d_+$. By construction, $\{d_-,a,b,c\}$ is
regular, so $c$ is the unique regular extension of $\{d_-,a,b\}$
larger than $b$. Since $d_+>c$, the quadruple
$\{d_-,a,b,d_+\}$ is irregular.
\end{proof}

By \eqref{glavnajednakost}, each of these quadruples can be extended to a
$D(u^2)$-quintuple by adjoining the nontriangular element $8u^2$. For each
fixed positive integer $k$, this gives infinitely many such quintuples with
smallest element $T_{16k^2-1}$ and distinct positive integer parameters $u$.

\begin{cor}
For every positive integer $k$, there exist infinitely many distinct
positive integers $u$ for which a $D(u^2)$-quintuple has smallest
element $T_{16k^2-1}$ and exactly four triangular elements.
\end{cor}

\section*{Acknowledgements}
The author is supported by the Croatian Science Foundation grant
no.~IP-2022-10-5008 and by the University of Split, grant
no.~IP-UNIST-44, funded by the European Union - NextGenerationEU.

\bigskip
\noindent University of Split, Faculty of Science,
Ru\dj{}era Bo\v{s}kovi\'{c}a 33, 21000 Split, Croatia

\noindent Email: \href{mailto:pradic@pmfst.hr}{pradic@pmfst.hr}


\begin{thebibliography}{99}

\bibitem{AHS1980}
J.~Arkin, V.~E.~Hoggatt, Jr., and E.~G.~Straus,
\textit{On Euler's solution to a problem of Diophantus---II},
Fibonacci Quart. \textbf{18} (1980), 170--176.

\bibitem{BagciZhouZheng}
S.~Bagchi and C.~Zhou-Zheng,
\textit{Diophantine $m$-tuples of triangular numbers}, arXiv:2608.27697 (2026).


\bibitem{BliznacTriangular}
M.~Bliznac Trebje\v{s}anin,
\textit{On Diophantine triples containing a triangular number},
Math. Slovaca (2026).

\bibitem{BliznacTriangular2}
M.~Bliznac Trebje\v{s}anin,
\textit{On Diophantine pairs and triples of triangular numbers}, arXiv:2603.29565
(2026).


\bibitem{BliznacFilipin2019}
M.~Bliznac Trebje\v{s}anin and A.~Filipin,
\textit{Nonexistence of $D(4)$-quintuples},
J. Number Theory \textbf{194} (2019), 170--217.

\bibitem{Dujella}
Y.~Bugeaud, A.~Dujella, and M.~Mignotte,
\textit{On the family of Diophantine triples
$\{k-1,k+1,16k^3-4k\}$},
Glasgow Math. J. \textbf{49} (2007), 333--344.

\bibitem{Deshpande}
M.~N.~Deshpande,
\textit{One property of triangular numbers},
Portugaliae Math. \textbf{55} (1998), 381--383.

\bibitem{knjiga}
A.~Dujella, \textit{Number Theory},
\v{S}kolska knjiga, Zagreb, 2021.

\bibitem{Fujita}
Y. Fujita, \textit{The extensibility of Diophantine pairs {k-1, k+1}}, J. Number Theory \textbf{128} (2008), 322-353

\bibitem{Gibbs2010}
P.~Gibbs, \textit{Adjugates of Diophantine quadruples},
Integers \textbf{10} (2010).

\bibitem{HeTogbeZiegler2019}
B.~He, A.~Togb\'{e}, and V.~Ziegler,
\textit{There is no Diophantine quintuple},
Trans. Amer. Math. Soc. \textbf{371} (2019), 6665--6709.

\end{thebibliography}
\end{document}